\documentclass[12pt,twoside]{amsart}
\usepackage{amsmath}
\usepackage{enumerate}
\usepackage{amsthm}
\usepackage{amsfonts}
\usepackage{amssymb}
\usepackage{latexsym}
\usepackage{mathrsfs}
\usepackage{amsmath}
\usepackage{amsthm}
\usepackage{amsfonts}
\usepackage{amssymb}
\usepackage{latexsym}
\usepackage{geometry}
\usepackage{dsfont}
\usepackage[dvips]{graphicx}
\usepackage{color}
\usepackage[all]{xy}

\date{}
\allowdisplaybreaks[4] \footskip=15pt
\renewcommand{\uppercasenonmath}[1]{}

\usepackage{graphicx,amssymb}
\usepackage[all]{xy}
\usepackage{amsmath}

\allowdisplaybreaks
\usepackage{amsthm}
\usepackage{color}

\theoremstyle{plain}
\newtheorem{theorem}{Theorem}[section]
\newtheorem{proposition}[theorem]{Proposition}
\newtheorem{lemma}[theorem]{Lemma}
\newtheorem{corollary}[theorem]{Corollary}
\theoremstyle{definition}
\newtheorem{example}[theorem]{Example}
\newtheorem{definition}[theorem]{Definition}

\theoremstyle{definition}
\newtheorem*{acknowledgement}{Acknowledgement}

\theoremstyle{remark}

\newcommand{\pf}{\noindent\begin {proof}}
\newcommand{\epf}{\end{proof}}

\newcommand{\Ext}{\mbox{\rm Ext}}
\newcommand{\Hom}{\mbox{\rm Hom}}
\newcommand{\Tor}{\mbox{\rm Tor}}

\def\m{{\frak m}}

\def\GV{{\rm GV}}
\def\tor{{\rm tor_{\rm GV}}}
\def\Hom{{\rm Hom}}
\def\Ext{{\rm Ext}}
\def\P{{\rm P}}
\def\Tor{{\rm Tor}}

\def\PF{{\rm PF}}

\def\Spec{{\rm Spec}}
\def\A{\mathcal{A}}
\def\B{\mathcal{B}}

\def\Mod{{\rm Mod}}
\def\E{{\rm E}}

\def\Nil{{\rm Nil}}

\def\NP{{\rm NP}}

\def\Z{{\rm Z}}

\def\GV{{\rm GV}}
\def\Max{{\rm Max}}
\def\DW{{\rm DW}}

\def\PvMR{{\rm PvMR}}
\def\PvMD{{\rm PvMD}}

\begin{document}
\begin{center}
{\large  \bf A new version of \P-flat modules and its applications}

\vspace{0.5cm}   Wei Qi$^{a}$, Xiaolei Zhang$^{a}$\\

{\footnotesize a.\  School of Mathematics and Statistics, Shandong University of Technology, Zibo 255049, China\\


Corresponding Author: Xiaolei Zhang, E-mail: zxlrghj@163.com\\}
\end{center}

\bigskip
\centerline { \bf  Abstract}
\bigskip
\leftskip10truemm \rightskip10truemm \noindent

In this paper, we introduce and study the class of $\phi$-$w$-\P-flat modules, which can be seen as generalizations of both $\phi$-\P-flat modules and $w$-\P-flat modules. In particular, we obtain that the class of $\phi$-$w$-\P-flat modules is covering. We also utilize the class of $\phi$-$w$-\P-flat modules to characterize $\phi$-von Neumann regular rings,  strong $\phi$-rings and $\phi$-PvMRs.
\vbox to 0.3cm{}

\leftskip0truemm \rightskip0truemm
\bigskip

Throughout this paper, all rings are commutative with identity and all modules are unitary.

It is well-known that the class of flat modules is one of the three classical classes of modules. Some important rings can be characterized by flat modules, such as von Neumann regular rings,  Pr\"{u}fer domains, and coherent rings. The generalizations of flat modules attracted many algebraists. For example, the notion of torsion-free modules can be seen as a generalization of that of flat modules over domains. The notion of \P-flat modules over commutative rings, introduced by Li and Wang \cite{LW04}, can also be seen as the generalization of that of torsion-free modules over domains. \P-flat modules are important in studying some rings. For example, a ring $R$ is von Neumann regular if and only if every $R$-module is \P-flat; a domain $R$ is Pr\"{u}fer if and only if every \P-flat module is flat (see Cheniour and Mahdou \cite{CM15}); a ring $R$ is a \PF\ ring if and only if every (principal) ideal is \P-flat.

The semi-star operation is a very important tool to extend various of commutative rings. The $w$-operation over domains was introduced by  Wang \cite{fm97} in 1997, and was extended to commutative rings by Yin et al. \cite{hfxc11} in 2010. Since then, using $w$-operations to construct new rings and modules is a fresh approach to algebraists. In 2014, Kim and Wang \cite{WK14} introduced and studied $w$-flat modules over commutative rings. In 2015, Wang and Qiao \cite{fq15} introduced the $w$-weak global dimensions of rings, and showed that PvMDs are exactly domains with $w$-weak global dimensions at most $1$.  The $w$-\P-modules recently introduced by Xia and Qiao \cite{X22} can be seen as a $w$-version of \P-flat modules. The authors \cite{X22} showed that a ring $R$ is a von Neumann regular ring if and only if every $R$-module is $w$-\P-flat; a domain $R$  is a DW-domain if and only if  any $w$-\P-flat module is \P-flat. They \cite{X22} also studied rings in which any  principal  ideal  is $w$-\P-flat.

The $\phi$-ring, introduced by Badawi \cite{A97} at the end of the last century,  is an important extension of   domains. Some classical results over domains can be generalized to $\phi$-rings. To generalize some results on $w$-operations over domains to $\phi$-ring, the second named author and Zhao \cite{ZZ21}  introduced $\phi$-$w$-flat modules over $\phi$-rings and $\phi$-PvMRs. In particular, the authors \cite{ZZ21} proved that a  strong $\phi$-ring $R$ is $\phi$-PvMR if and only if  any  any $R$-ideal is $\phi$-$w$-flat. However, the natural question is:
\begin{center}
\textbf{A $\phi$-ring $R$ is a $\phi$-PvMR if and only if  any   ideal of is $\phi$-$w$-flat?}
\end{center}
To answer this question, we generalized $w$-\P-flat modules to get a new class of modules: $\phi$-$w$-\P-flat modules. We utilize the class of $\phi$-$w$-\P-flat modules to characterize $\phi$-von Neumann  regular rings,  domains, $\DW$-rings and strong $\phi$-rings. Specially, we prove a $\phi$-ring $R$ is $\phi$-PvMR and a  strong $\phi$-ring  if and only if  any  ideal $R$  is $\phi$-$w$-flat, and so give a negative answer to the above question.

\section{$\phi$-$w$-$P$-flat modules and their basic properties}

First, we recall some baisc knowledge on $w$-operations (see \cite{fk16} for more details).  Let $R$ be a ring and  $J$ be a  finitely generated ideal of $R$.  If the multiplicative map $R\rightarrow \Hom_R(J,R)$  is an isomorphism,  then $J$ is called to be a $\GV$-ideal. The set of all $\GV$-ideals is denoted by $\GV(R)$.  Suppose $M$ is an $R$-module. Denote by
\begin{center}
$\tor(M):=\{x\in M\mid\ \mbox{there is}\ J\in \GV\ \mbox{ such that }\ Jx=0\}.$
\end{center}
If $\tor(M)=M$ (resp.,  $\tor(M)=0$),  then $M$ is said to be $\GV$-torsion  (resp., $\GV$-torsion-free).  Now, suppose $M$ is a $\GV$-torsion-free $R$-module.   If  $\Ext_R^1(R/J,M)$ $=0$ for  any $J\in \GV(R)$,  then $M$ is called a $w$-module. The $w$-envelope of a $\GV$-torsion-free $M$ is defined as follows:
\begin{center}
{\rm $M_w:=\{x\in \E_R(M)\mid\ \mbox{there is }\ J\in \GV\ \mbox{such that}\ Jx\subseteq M\}$,}
\end{center}
where $\E_R(M)$ is the injective envelope of $M$. Obviously, a $\GV$-torsion-free module $M$ is a $w$-module  if and only if $M_w=M$.  If every $R$-module  is $w$-module,  then $R$ is a $\DW$-ring.
The maximal $w$-ideal  is the maximal $w$-submodule of $R$. The set of all $w$-ideals is denoted by $w$-$\Max(R)$. Notice that $w$-$\Max(R)$ is a nonempty of $\Spec(R)$. An $R$-homomorphism $f:M\rightarrow N$ is called a $w$-monomorphism (Resp.,
$w$-epimorphism, $w$-isomorphism), if $f_\m:M_\m\rightarrow N_\m$ is a monomorphism (Resp., an epimorphism, an isomorphism) as $R_\m$-modules for  any  $\m\in w$-$\Max(R)$.  If, for  any  $\m\in w$-$\Max(R)$, $A_\m\rightarrow B_\m\rightarrow C_\m$ is exact as $R_\m$-modules,  then the
$R$-sequence $A\xrightarrow{f} B\xrightarrow{g} C$ is said to be $w$-exact.

Recall from \cite{CM13}, an $R$-module $M$ is said to be \P-flat, if  $(r, m)\in R\times M$ with $rm = 0$ implies $m\in(0 :_R r)M$.  Suppose $r\in R$.  Then  by \cite[Proposition 1]{H60},  $$\Tor^R_1(M, R/rR)=\{m \in M|rm=0\}/(0 :_R r)M.$$ Hence, an $R$-module $M$ is \P-flat if and only if   $\Tor_1^R(R/Rr,M)=0$ for  any $r\in R$. Recall from \cite{X22}, an $R$-module $M$ is said to be $w$-\P-flat, if  for  any  satisfying $rm = 0$ with $(r, m)\in R\times M$
, there exists $J\in\GV(R)$ such that $Jm\subseteq (0 :_R r)M$.  Hence, an $R$-module $M$ is $w$-\P-flat if and only if   $\Tor_1^R(R/Rr,M)$ is $\GV$-torsion for  any $r\in R$.

\begin{definition}\label{naga-lucas}
Let $R$ be a  ring and $M$ be an $R$-module.   If $\Tor_1^R(R/Ra,M)$ is $\GV$-torsion for  any  non-nilpotent element $a\in R$,  then $M$ is said to be $\phi$-$w$-\P-flat.
\end{definition}
Trivially, $\GV$-torsion modules and $w$-\P-flat modules are $\phi$-$w$-\P-flat. Next, we will give a characterization of $\phi$-$w$-\P-flat modules.
\begin{theorem}\label{c-t-pf-l}
Let $M$ be an $R$-module.  Then the following statements are equivalent:
\begin{enumerate}[(1)]
    \item $M$ is $\phi$-$w$-\P-flat;
    \item  for  any  non-nilpotent $a\in R$, the natural $R$-homomorphism $Ra\otimes_RM\rightarrow R\otimes_RM$ is a $w$-monomorphism;
    \item  for  any  non-nilpotent element $a$ satisfying  $am=0$  with $(a,m)\in R\times M$, there exists  an ideal  $J\in \GV(R)$ such that $Jm\subseteq (0:_Ra)M$.
\end{enumerate}
\end{theorem}

\begin{proof} $(1)\Leftrightarrow (2)$  Let $a$ be a  non-nilpotent element in $R$.  Then there exists an exact sequence $0=\Tor_1^R(R,M)\rightarrow \Tor_1^R(R/Ra,M)\rightarrow Ra\otimes_RM\rightarrow R\otimes_RM$.  So $M$ is $w$-\P-flat if and only if  the natural $R$-homomorphism $Ra\otimes_RM\rightarrow R\otimes_RM$ is $w$-monomorphism.

$(1)\Leftrightarrow (3)$:  It follows by the following natural isomorphism (\cite[Proposition 1]{H60}): $$\Tor_1^R(R/Ra,M)\cong \{m\in M\mid am = 0\}/(0:_Ra)M.$$
\end{proof}

\begin{proposition}\label{c-t-pf-2}
 Suppose $0\rightarrow A\rightarrow B\rightarrow C\rightarrow 0$ is $w$-exact sequence of $R$-modules.  If $A$ and $C$ are $\phi$-$w$-\P-flat,  then $B$  is also $\phi$-$w$-\P-flat.
\end{proposition}
\begin{proof} Suppose $a$ is a non-nilpotent element in $R$.  Then there is a $w$-exact sequence  $$\Tor_1^R(R/Ra,A)\rightarrow \Tor_1^R(R/Ra,B)\rightarrow \Tor_1^R(R/Ra,C).$$ Since $A$ and $C$ are $\phi$-$w$-\P-flat, $\Tor_1^R(R/Ra,A)$ and $\Tor_1^R(R/Ra,C)$ are $\GV$-torsion.  So, $\Tor_1^R(R/Ra,B)$ is also $\GV$-torsion.  Hence, $B$ is $\phi$-$w$-\P-flat.
\end{proof}

\begin{proposition}\label{c-t-pf-2}
 Suppose $\{M_i\mid i\in\Lambda\}$ is a directed system of $\phi$-$w$-\P-flat modules.  Then $\lim\limits_{\longrightarrow}M_i$ is a $\phi$-$w$-\P-flat  module.
\end{proposition}
\begin{proof} Suppose $a$ is  a non-nilpotent element in $R$, $\{M_i\mid i\in\Lambda\}$ is a directed system of $\phi$-$w$-\P-flat modules.   Then  $\Tor_1^R(R/Ra,M_i)$  is $\GV$-torsion  for  any $i\in\Lambda$.  So $\Tor_1^R(R/Ra,\lim\limits_{\longrightarrow}M_i)\cong \lim\limits_{\longrightarrow}\Tor_1^R(R/Ra,M_i)$ is also $\GV$-torsion.  Hence $\lim\limits_{\longrightarrow}M_i$ is $\phi$-$w$-\P-flat.
\end{proof}

\begin{proposition}\label{c-t-pf-2}
 Suppose $\{M_i\mid i\in\Lambda\}$ is a set of $\phi$-$w$-\P-flat modules.   Then $\bigoplus\limits_{i\in\Lambda}M_i$ is $\phi$-$w$-\P-flat if and only if   $M_i$ is $\phi$-$w$-\P-flat for  any $i\in\Lambda$.
\end{proposition}
\begin{proof} Suppose $a$ is a non-nilpotent element in $R$. It follows by $$\Tor_1^R(R/Ra,\bigoplus\limits_{i\in\Lambda}M_i)\cong \bigoplus\limits_{i\in\Lambda}\Tor_1^R(R/Ra,M_i).$$
\end{proof}

\begin{proposition}\label{c-t-pf-2}
The pure submodules and pure quotients of $\phi$-$w$-\P-flat modules  are $\phi$-$w$-\P-flat.
\end{proposition}
\begin{proof} Suppose $a$ is  a non-nilpotent element in $R$.  Then we have an exact sequence $0\rightarrow Ra \rightarrow R\rightarrow R/Ra \rightarrow 0$.  Suppose $0\rightarrow M\rightarrow N\rightarrow L\rightarrow 0$ is pure exact sequence of $R$-modules with $N$  $\phi$-$w$-\P-flat. Then we have the following commutative diagram of  exact sequences:
$$\xymatrix{
   0 \ar[r]^{} & M\otimes_R Ra \ar[d]_{f}\ar[r]^{} & N\otimes_R Ra  \ar@{>->}[d]_{}\ar[r]^{} & L\otimes_R Ra \ar[d]_{g}\ar[r]^{} &  0\\
   0 \ar[r]^{} & M\otimes_R R \ar@{->>}[d]_{}\ar[r]^{} & N\otimes_R R  \ar@{->>}[d]_{}\ar[r]^{} & L\otimes_R R \ar@{->>}[d]_{}\ar[r]^{} &  0\\
   0 \ar[r]^{} & M\otimes_R  R/Ra\ar[r]^{} & N\otimes_R  R/Ra  \ar[r]^{} & L\otimes_R  R/Ra \ar[r]^{} &  0\\}$$
It follows by Snake Lemma that $f: M\otimes_R Ra \rightarrow M\otimes_R R$ and $g: L\otimes_R Ra\rightarrow L\otimes_R R$ are $w$-monomorphisms. It follows  by  Theorem \ref{c-t-pf-l} that $M$ and $L$ are $\phi$-$w$-\P-flat.
\end{proof}

Let $\mathcal{F}$ be a class of $R$-modules and $M$ be an $R$-module. Recall from \cite[Definition  5.1]{gt} that if there  exists an $R$-homomorphism $f: F(M)\rightarrow M$ with $F(M)\in \mathcal{F}$,
such that for  any $F\in \mathcal{F}$, $g: F\rightarrow M$ there exists $h: F\rightarrow F(M)$ such that the following diagram is commutative:
$$\xymatrix{
  &&F \ar[rd]^g  \ar@{.>}[ld]_h  \\
&F(M) \ar[rr]^f &&M
              }$$
Then  $f: F(M)\rightarrow M$ is called an $\mathcal{F}$-precover of $M$. Furthermore, if  the endomorphism $h$ of $F(M)$ satisfying $f=f\circ h$ is an automorphism,  then $f: F(M)\rightarrow M$ is called the $\mathcal{F}$- cover of $M$.  Dually, one can define the  $\mathcal{F}$-preenvelope and $\mathcal{F}$-envelope of $R$-modules.

Let $\A$ and $\B$ be classes of  $R$-modules. Denote by $\A^{\perp_1}=\{M\in R$-$\Mod\mid \Ext_R^1(F,M)=0,$  for  any  $F\in \A\}$, $^{\perp_1}\A=\{M\in R$-$\Mod\mid \Ext_R^1(M,F)=0$ for  any  $F\in \A\}.$   If $\A=^{\perp_1}\B$ and $\B=\A^{\perp_1}$,  then $(\A,\B)$ is called a cotorsion pair.  If  any $R$-module has an $\A$-cover and a $\B$-envelope,  then $(\A,\B)$ is said to be perfect. It is well-known that the class of flat modules is covering. The authors in
\cite{Z19,ZZW21} obtained the classes of $w$-flat modules and $\phi$-flat modules are covering. The next result shows that the class of  $\phi$-$w$-\P-flat modules is also a covering.
\begin{theorem}\label{gv-f-cover}
Let $R$ be  a  ring. Denote by $\phi$-$w$-$\mathscr{F}_{P}$ the class of $\phi$-$w$-\P-flat modules.  Then $(\phi$-$w$-$\mathscr{F}_{P}$,$\phi$-$w$-$\mathscr{F}_{P}^{\perp_1})$ is a perfect cotorsion pair.  Hence,  any $R$-module has a $\phi$-$w$-\P-flat cover.
\end{theorem}
\begin{proof} Obviously, $R$  is  $\phi$-$w$-\P-flat, and $\phi$-$w$-$\mathscr{F}_{P}$ is closed under extension and direct summand. It follows by Lemma \ref{c-t-pf-2}  and \cite[Theorem  3.4]{HJ08} that
$(\phi$-$w$-$\mathscr{F}_{P}$,$\phi$-$w$-$\mathscr{F}_{P}^{\perp_1})$ is a perfect cotorsion pair.  So,   any $R$-module has a $\phi$-$w$-\P-flat cover.
\end{proof}

\section{$\phi$-rings determined by $\phi$-$w$-\P-flat modules}

Let $R$ be a ring. Denote by $\Nil(R)$ the nil-radical of  $R$, and $\Z(R)$ the set of all zero-divisors of  $R$. If $\Nil(R)$ is a prime ideal,  then $R$ is called an
$\NP$-ring.   If, for  any  $x\not\in P$, we have $P\subsetneq (x)$,  then the prime ideal $P$ is called a divisible prime ideal (see  Badawi \cite{A97}).  If $P$ is a divisible prime ideal,  then  for  any  ideal  $I$ of $R$, we have $I\subseteq P$, or $P\subseteq I$.
Write
\begin{center}
$\mathcal {H}=\{R|R$  is a ring and $\Nil(R)$  is  a  divisible  prime  ideal of $R\}$.
\end{center}
If  $R\in \mathcal {H}$,  then  $R$  is called a $\phi$-ring. Trivially,  any  domain  is a $\phi$-ring.  Furthermore, if $\Z(R)=\Nil(R)$,  then a $\phi$-ring $R$ is called a strong $\phi$-ring. We will investigate $\phi$-$w$-\P-flat modules over $\phi$-rings, and generalize some results of $w$-\P-flat modules over domains to $\phi$-rings.

Recall from \cite{ZWT13}, a $\phi$-ring $R$ is called a $\phi$-von Neumann  regular ring,  provided  any $R$-module  is $\phi$-flat.  It follows by
\cite[Theorem 3.1]{ZZ21} that a $\phi$-ring $R$ is $\phi$-von Neumann  regular   if and only if  any $R$-module  is $\phi$-$w$-flat,  if and only if  $a\in (a^2)_w$ for any  non-nilpotent element $a\in R$,  if and only if $R/\Nil(R)$ is a domain with Krull dimension $0$, that is, $R/\Nil(R)$ is a field. The following gives a new characterization of $\phi$-von Neumann  regular  rings.
\begin{theorem}\label{phi-vn}A $\phi$-ring $R$ is $\phi$-von Neumann  regular   if and only if  any $R$-module  is $\phi$-$w$-\P-flat.
\end{theorem}
\begin{proof}Obviously, if a $\phi$-ring $R$ is a $\phi$-von Neumann  regular  ring,  then  any $R$-module  is $\phi$-flat, and so  is $\phi$-$w$-\P-flat. On the other hand,  let $a$ be a non-nilpotent element in $R$.  Then $R/Ra$ is $\phi$-$w$-\P-flat. So, $\Tor_1^R(R/Ra,R/Ra)$  is $\GV$-torsion, that is, $Ra/Ra^2$  is $\GV$-torsion.  So  $a\in Ra\subseteq (Ra)_w=(Ra^2)_w$.  It follows by
\cite[Theorem  3.1]{ZZ21} that $R$ is a $\phi$-von Neumann  regular ring.
\end{proof}

It worth to notice that Theorem \ref{phi-vn} does not hold for all rings. For example, let $R$ be a non-field von Neumann regular ring.  Then $R$  is not a $\phi$-ring.  So it is also not  $\phi$-von Neumann  regular. However,  any $R$-module  is flat, and so  is $\phi$-$w$-\P-flat.

It follows by \cite[Theorem  1.9]{ZZ21} that a $\phi$-ring $R$ is a domain  if and only if  any $\phi$-flat module is flat,  if and only if  any $\phi$-$w$-flat is $w$-flat. Similar, we have the following result.
\begin{theorem}\label{domain}A $\phi$-ring $R$ is a domain  if and only if  any $\phi$-$w$-\P-flat $R$-module  is $w$-\P-flat.
\end{theorem}
\begin{proof} Obviously, if a $\phi$-ring $R$ is  domain,  then  any $\phi$-$w$-\P-flat $R$-module  is $w$-\P-flat. On the other hand, it follows by
\cite[ Proposition  1.7]{ZZ21}, $R/\Nil(R)$ is $\phi$-flat $R$-module,  so is a $\phi$-$w$-\P-flat $R$-module. By assumption, $R/\Nil(R)$  is $w$-\P-flat. Let $a$ be a nilpotent element in $R$.  Then $\Tor_1^R(R/(a),R/\Nil(R))\cong(a)\cap\Nil(R)/a\Nil(R)=(a)/a\Nil(R)$  is $\GV$-torsion. Hence, there  exists  a $\GV$-ideal  $J$ such that $aJ\subseteq a\Nil(R)$.  Since $J$  is  nonnil  ideal, it follows  by \cite[Lemma 1.6]{ZZ21} that $\Nil(R)=J\Nil(R)$. So $aJ\subseteq a\Nil(R)=aJ\Nil(R)\subseteq aJ$, that is, $aJ=aJ\Nil(R)$.  Because  $aJ$  is  finitely generated, it follows by Nakayama Lemma that $aJ=0$.  Since  $J\in \GV(R)$, $a\in R$  is  $\GV$-torsion-free, we have $a=0$. Consequently, $\Nil(R)=0$. So $R$ is a reduced $\phi$-ring, that is, a  domain.
\end{proof}

It is worth to notice that  Theorem \ref{domain} is not true for commutative rings. For example, let $R$ be a non-field von Neumann regular  ring. Then  any $R$-module  is $w$-\P-flat, and so  is $\phi$-$w$-\P-flat. Hence any $\phi$-$w$-\P-flat $R$-module  is $w$-\P-flat. However, $R$  is not a  domain.

\begin{corollary} A $\phi$-ring $R$ is a  $\DW$-domain  if and only if  any $\phi$-$w$-\P-flat $R$-module  is \P-flat.
\end{corollary}
\begin{proof} Obviously,  if $\phi$-ring $R$ is a $\DW$-domain,  then  any $\phi$-$w$-\P-flat $R$-module  is \P-flat. On the other hand, it follows by  Theorem
\ref{domain} that $R$ is a domain. In this situation,  any $w$-\P-flat $R$-module  is \P-flat.   It follows by \cite[ Theorem  4]{X22} that $R$ is a $\DW$-domain.
\end{proof}

\begin{lemma}\label{P1X22}\cite[Proposition  1]{X22}
 Every cyclic $w$-\P-flat module is $w$-flat.
\end{lemma}

If  any  non-nilpotent element in a $\phi$-ring $R$ is non-zero-divisor,  then  $R$ is called a strong $\phi$-ring. The following result gives a characterization of  strong $\phi$-rings.
\begin{theorem}\label{Sphi} Suppose $R$ is $\phi$-ring.  Then  the following statements are equivalent:
\begin{enumerate}[(1)]
    \item $R$ is a strong $\phi$-ring;
    \item  any ideal of $R$  is $\phi$-$w$-\P-flat;
    \item  any  principal ideal of $R$ is $\phi$-$w$-\P-flat;
 \item  any  nonnil  principal ideal of $R$ is $w$-flat;
 \item any submodule of $\phi$-$w$-\P-flat $R$-module is $\phi$-$w$-\P-flat;
 \item  for  any  element $r\in R$ and any non-nilpotent element $a\in R$ with $ra=0$, there exist $b\in (0:_RRa)$ and $J\in\GV(R)$ such that $Jr\subseteq Rbr$.
\end{enumerate}
\end{theorem}
\begin{proof}
  $(2)\Rightarrow (3)$ and $(5)\Rightarrow (2)$: obvious.

 $(1)\Rightarrow (2)$: Suppose $R$ is a strong $\phi$-ring. Let $I$ be an ideal of $R$.   For  any non-nilpotent element  $a$ with $(a,i)\in R\times I$ satisfying  $ai=0$, we have $i=0$.  So $0=Ji\subseteq (0:_Ra)I$ for  any  ideal  $J\in \GV(R)$.  It follows by  Theorem \ref{c-t-pf-l} that $I$ is $\phi$-$w$-\P-flat.

 $(3)\Rightarrow (4)$:   Suppose $Ra$ is a nonnil  principal  ideal of $R$.  Then $a$ is  non-nilpotent in $R$.  Suppose $I$ is a principal  ideal of $R$. Then we have a natural isomorphism
 $$\Tor_1^R(R/Ra,I)\cong \Tor_2^R(R/Ra,R/I)\cong \Tor_1^R(R/I,Ra).$$  By assumption, $I$ is $\phi$-$w$-\P-flat. So $\Tor_1^R(R/I,Ra)$  is $\GV$-torsion.  Hence  $Ra$ is cyclic and $w$-\P-flat. It follows  by Lemma \ref{P1X22}  that $Ra$ is $w$-flat.

$(4)\Rightarrow (5)$:  Suppose $M$ is a $\phi$-$w$-\P-flat module and $N$ is a submodule of $M$.  Suppose $a$ is a non-nilpotent element in $R$. Then $Ra$ is $w$-flat.
Consider the long exact sequence $$\Tor_2^R(R/Ra,M/N)\rightarrow \Tor_1^R(R/Ra,N)\rightarrow \Tor_1^R(R/Ra,M).$$  Since $M$ is $\phi$-$w$-\P-flat, then $\Tor_1^R(R/Ra,M)$ is $\GV$-torsion. To show $N$ is $\phi$-$w$-\P-flat, we only need to prove $\Tor_2^R(R/Ra,M/N)$  is $\GV$-torsion. In fact, consider the exact sequence $$0=\Tor_2^R(R,M/N)\rightarrow \Tor_2^R(R/Ra,M/N)\rightarrow \Tor_1^R(Ra,M/N)\rightarrow \Tor_1^R(R,M/N)=0.$$  Since $Ra$ is $w$-flat,  $\Tor_1^R(Ra,M/N)$ is $\GV$-torsion.  So  $\Tor_2^R(R/Ra,M/N)$ is $\GV$-torsion.

$(3)\Rightarrow (6)$:  Suppose $r\in R$ and $a$ is a  non-nilpotent element in $R$.  Then $\Tor_1^R(R/Ra,Rr)$ is $\GV$-torsion. It follows by \cite[ Proposition  1]{H60} that $$\Tor_1^R(R/Ra,Rr)\cong \{xr\in Rr\mid axr = 0\}/(0:_Ra)Rr.$$
 So  if $ra=0$,  then there exists $\GV$-ideal $J$ such that $Jr\subseteq (0:_Ra)Rr$.  Since $J$ is  finitely generated,  there exists $b\in (0:_RRa)$ such that $Jr\subseteq Rbr$.


$(6)\Rightarrow (1)$:  Suppose $a$ is a zero-divisor in $R$.  Then there exists $0\not=r\in R$ such that $ar=0$. By converse, assume $a$  is not nilpotent.  Then there  exist $b\in (0:_RRa)$ and a $\GV$-ideal $J$ such that  $Jr\subseteq Rbr$ by assumption. Since $R$ is a $\phi$-ring and $a$  is not  nilpotent,  then $b$  is  nilpotent. Assume $b^m=0$.  Then $$J^mr=J^{m-1}(Jr)\subseteq J^{m-2}J(Rbr)\subseteq J^{m-3}J(Rb^2r)\subseteq \cdots \subseteq J(Rb^{m-1}r)\subseteq Rb^mr=0.$$
Notice that $J^m$ is a $\GV$-ideal.  So $J^m$ is a finitely generated semi-regular ideal. Hence $r=0$, a contradiction.
\end{proof}

Let $R$ be a  $\phi$-ring.  Denote by $\phi:R\rightarrow R_{\Nil(R)}$ with $\phi(r)=\frac{r}{1}$ for any $r\in R$. Recall from \cite{kf12}, an ideal $I$ of $R$ is called $\phi$-$w$-invertible,  if $(\phi(I)\phi(I)^{-1})_W=\phi(R)$, where $W$ is the $w$-operation of $\phi(R)$.
Recall from \cite[ Definition  3.2]{ZZ21} that  if  any  finitely generated nonnil  ideal of $R$  is $\phi$-$w$-invertible,  then 
$\phi$-ring $R$ is called a $\phi$-$\PvMR$. It follows by \cite[Theorem  3.3]{ZZ21} that a  strong $\phi$-ring $R$ is a $\phi$-$\PvMR$ if and only if  any  submodule of $w$-flat module is $\phi$-$w$-flat,  if and only if   any  submodule of  flat  module  is $\phi$-$w$-flat,  if and only if any  ideal of $R$  is $\phi$-$w$-flat,  if and only if  any  nonnil  ideal of $R$ is $\phi$-$w$-flat,  if and only if any  finite type  nonnil  ideal of $R$  is $\phi$-$w$-flat. Then a natural question is that:
\begin{center}
Can we delete the condition ``$R$ is  strong $\phi$-ring '' is in \cite[Theorem  3.3]{ZZ21}?
\end{center}
The following result give a negative answer to this qustion.

\begin{corollary}\label{pvmr} Let $R$ be a $\phi$-ring. Then  the following statements are equivalent:
\begin{enumerate}[(1)]
    \item $R$ is a $\phi$-$\PvMR$ and a strong $\phi$-ring;
    \item any  submodule of $w$-flat $R$-module  is $\phi$-$w$-flat;
    \item any  submodule of flat $R$-module is $\phi$-$w$-flat;
    \item  any  ideal of $R$ is $\phi$-$w$-flat;
 \item any  nonnil  ideal of $R$ is $\phi$-$w$-flat;
 \item  any  finite type  nonnil  ideal of $R$ is $\phi$-$w$-flat.
\end{enumerate}
\end{corollary}
\begin{proof}$(1)\Rightarrow (2)$: Follows by \cite[ Theorem  3.3]{ZZ21}.

$(2)\Rightarrow (3)\Rightarrow (4)\Rightarrow (5)\Rightarrow (6)$: Trivial.

$(6)\Rightarrow (1)$:  Suppose  every finite type  nonnil  ideal of $R$  is $\phi$-$w$-flat, and so  is $\phi$-$w$-\P-flat.   Hence every nonnil principal ideal of $R$  is $\phi$-$w$-flat, and thus $w$-flat  by Lemma \ref{P1X22}. It follows by Theorem \ref{Sphi} that $R$ is a strong $\phi$-ring.  So, it follows by \cite[ Theorem  3.3]{ZZ21} that $R$ is a  $\phi$-$\PvMR$.
\end{proof}

\begin{example}\label{w-phi-prufer}  Suppose $D$ is a non-field $\PvMD$, $K$ is its quotient field. Since $K/D$  is divisible $D$-module,  we have the  trivial extension $R:=D(+)K/D$  is $\phi$-ring, but not  a strong  $\phi$-ring $($ see \cite[Remark 1]{FA05}$).$  Then $\Nil(R)=0(+)K/D$.  So  $R/\Nil(R)\cong D$  is a $\PvMD$.
It follows  by \cite[Theorem  3.3]{ZZ21} that $R$ is a $\phi$-$\PvMR$. However,  by  Corollary \ref{pvmr},  there exists an ideal $I$ of $R$ wihch is not  $\phi$-$w$-flat.
\end{example}

\begin{acknowledgement}\quad\\
The first author was supported by  the National Natural Science Foundation of China (No.12201361) and Doctoral Research Initiation Fund of Shandong University of Technology (No.423002). The second author was supported by  the Doctoral Research Initiation Fund of Shandong University of Technology (No.422030).
\end{acknowledgement}


\bigskip
\begin{thebibliography}{99}
\bibitem{FA05} D. F. Anderson, A. Badawi, On $\phi$-Dedekind rings and
$\phi$-Krull rings.  Houston J. Math.,  2005, 31:1007-1022.

\bibitem{A97}  A. Badawi,   On divided commutative rings.   Comm. Algebra, 1999, 27: 1465-1474.
\bibitem{CM15}  F.  Cheniour, N. Mahdou, On some Flatness properties over Commutative Rings. Acta
Math. Hungar, 2015, 146(1): 142-152.
\bibitem{CM13}  F.  Cheniour, N. Mahdou, When every principal ideal is flat. Portugaliae Math., 2013, 70:
51-58.


\bibitem{gt} R. Gobel,  J. Trlifaj, Approximations and endomorphism algebras of modules. De Gruyter Exp. Math., vol.  {\bf 41}, Berlin: Walter de Gruyter GmbH \& Co. KG, 2012.
\bibitem{H60}  A. Hattori, A Foundation of Torsion Theory for Modules over General Rings. Nagoya
Math. J., 1960, 17: 147-158.
\bibitem{HJ08} H. Holm, P. Jorgensen, Covers, precovers, and purity.  Illinois J. Math., 2008, 52(2): 691-703.
\bibitem{LW04}  S. S. Li, M. Y. Wang , On \P-flat modules.  J. Guangxi Norm. Univ. Nat. Sci, 2004,22(4): 36-40.

\bibitem{kf12}   H. Kim, F. G. Wang,  On $\phi$-strong Mori rings.  Houston J. Math., 2012,  38(2): 359-371.

\bibitem{kmo24}  H. Kim.,  N. Mahdou., E. Oubouhou,  Nonnil-\P-coherent rings and nonnil-\P\P-rings, to appear in Bull. Korean Math. Soc.
\bibitem{WK14} F. G. Wang, H. Kim, $w$-Injective modules and $w$-semi-hereditary rings.  J. Korean Math. Soc., 2014, {51}: 509-525.
\bibitem{fk16} F. G. Wang,  H. Kim,   Foundations of Commutative Rings and Their Modules.  Singapore: Springer, 2016.
\bibitem{fm97} F. G. Wang, R. L. Mccasland,  On $w$-modules over strong Mori domains.  Commun. Algebra, 1997, {25}(4): 1285-1306.

\bibitem{fq15} F. G. Wang, L. Qiao,   The $w$-weak global dimension of commutative rings.  Bull. Korean Math. Soc., 2015, 52(4): 1327-1338.
\bibitem{X22}  W. H. Xiao, L. Qiao, F. F. Song, $w$-\P-Flat Modules over Commutative Rings and Their applications. Journal of Jilin University (Science Edition), 2022,  60(2):269-276.
\bibitem{hfxc11}  H. Y. Yin,  F. G. Wang, X. S. Zhu , Y. H. Chen, $w$-modules over commutative rings.  J. Korean Math. Soc., 2011, {48}(1): 207-222.



\bibitem{Z19} X. L. Zhang,  Covering and Enveloping on $w$-operation. J. Sichuan Norm. Univ. Nat. Sci, 2019, 42(03):382-386.
\bibitem{ZZ21}  X. L. Zhang, W. Zhao, On $w$-$\phi$-flat modules and their homological dimensions.  Bull. Korean Math. Soc., 2021, 58 (4):1039-1052.
\bibitem{ZZW21} X. L. Zhang, W. Zhao, F. G. Wang, On $\phi$-flat Cotorsion Theory. J. Guangxi Norm. Univ. Nat. Sci, 2021,39(02):119-124.

\bibitem{ZWT13} W. Zhao,  F. G. Wang, G. H. Tang,    On $\phi$-von Neumann regular rings.  J. Korean Math. Soc., 2013, 50(1): 219-229.



\end{thebibliography}
\end{document}